\documentclass[a4paper,12pt]{article}

\usepackage{amsthm}   \usepackage{amsfonts}     \usepackage{amssymb}
\usepackage{amsmath}  \usepackage{mathrsfs}     \usepackage{mathtools}
\usepackage{url}       \usepackage{fancyhdr}

\usepackage{relsize}
\usepackage[all]{xy}    \usepackage{amscd}   \usepackage{epsfig}

\usepackage{array}   \usepackage{multirow}  \usepackage{booktabs,caption,fixltx2e}
\usepackage[flushleft]{threeparttable}

\usepackage{cases} 
\usepackage{ulem}  
\usepackage{xcolor}  

\usepackage[T1]{fontenc} 
\usepackage[utf8]{inputenc}
\usepackage{authblk}  
\usepackage{lineno}

\newtheorem{theorem}{Theorem}
\newtheorem{lemma}[theorem]{Lemma}
\newtheorem{proposition}[theorem]{Proposition}
\newtheorem{corollary}[theorem]{Corollary}

\theoremstyle{definition}
\newtheorem{example}[theorem]{Example}

\newtheorem{remark}{Remark}
\newtheorem*{case}{Case}
\newtheorem*{subcase}{Subcase}

\newcommand{\Aut}{{\mathrm{Aut}}}
\newcommand{\Inn}{{\mathrm{Inn}}}
\newcommand{\Out}{{\mathrm{Out}}}
\newcommand{\Mon}{{\mathrm{Mon}}}

\newcommand{\normal}{\trianglelefteq}

\newcommand{\la}{\langle}
\newcommand{\ra}{\rangle}

\begin{document}
\title{Totally symmetric dessins with nilpotent automorphism groups of class three}
\author[1,2]{Na-Er Wang\thanks{wangnaer@zjou.edu.cn}}
\author[3,4]{Roman Nedela\thanks{nedela@savbb.sk}}
\author[1,2]{Kan Hu\thanks{hukan@zjou.edu.cn}}
\affil[1]{School of Mathematics, Physics and Information Science, Zhejiang Ocean University, Zhoushan, Zhejiang 316022, People's Republic of China}
\affil[2]{Key Laboratory of Oceanographic Big Data Mining \& Application of Zhejiang Province, Zhoushan, Zhejiang 316022, People's Republic of China}
\affil[3]{Faculty of Natural Sciences, Matej Bel University, Tajovsk\'eho 40, 974 01, Bansk\'a Bystrica, Slovak Republic}
\affil[4]{Institute of Mathematics and Computer Science, Slovak Academy of Sciences, Bansk\'a Bystrica, Slovak Republic}
\maketitle
\begin{abstract}
A dessin is a 2-cell embedding of a connected bipartite graph into an orientable closed surface. An automorphism of a dessin is a permutation of the edges of the underlying graph which preserves the colouring of the vertices and extends to an orientation-preserving self-homeomorphism of the supporting surface. A dessin is regular if its automorphism group is transitive on the edges, and a regular dessin is totally symmetric if it is invariant under all dessin operations.  Thus totally symmetric dessins possesses the highest level of external symmetry. 
In this paper we present a classification of totally symmetric dessins with a nilpotent automorphism group of class three.\\[2mm]
\end{abstract}

\section{Introduction}
 By Bely\v{\i}'s theorem~\cite{Belyi1979}, a compact Riemann surface $S$, regarded as an algebraic curve, is definable over the field $\mathbb{\bar Q}$ of algebraic numbers if and only if  there is a non-constant meromorphic function $\beta$ from $S$ to the Riemann sphere with at most three critical values. By composing $\beta$ with a M\"obius transformation when necessary, these three critical points can be chosen as $\{0,1,\infty\}$. Joining $0$ and $1$ with a single edge we obtain the trivial bipartite map on the Riemann sphere, which lifts along $\beta$ to a bipartite map on $S$: The embedded bipartite graph is  the preimage $\beta^{-1}[0,1]$ of the closed interval $[0,1]$, and the black and the white vertices are the preimages of $0$ and $1$ respectively.

The absolute Galois group $\mathbf{G}:={\rm Gal}(\mathbb{\bar Q}/\mathbb{Q})$ has a natural action on the coefficients of polynomials and rational functions which define the Riemann surfaces and Bely\v{\i} functions. In particular, it induces an action of $\mathbf{G}$ on the bipartite maps. Grothendieck showed that this action is faithful~\cite{Grothendieck1997}, so we have a combinatorial approach to $\mathbf{G}$.
 Following Grothendieck such bipartite maps will be called dessins.

A dessin $D$ is regular if the associated Bely\v{\i} function $\beta$ is a regular covering, in which case the automorphism group $\Aut(D)$ of $D$, the group of covering transformations of $\beta$, acts transitively on the edges of $D$. Recently, Gonz\'alez-Diez and Jaikin-Zapirain have shown that the action of $\mathbf{G}$ on dessins remains faithful when restricted to regular dessins~\cite{GJ2013}.

An important problem in this field is the classification of regular dessins, typically imposing certain conditions on the supporting surfaces, the embedded graphs or the underlying automorphism groups~\cite{BJ2001,CJSW2013,CJSW2009,Du,HNW2014,
HNW2014-3,Jones2010,Jones2013,JSW2007,MNS2012}.

In this paper we investigate regular dessins with nilpotent automorphism groups, with emphsis on totally symmetric ones, i.e., regular dessins which are invariant under all dessin operations. As shown by Jones in \cite{Jones2013}, every nilpotent regular dessin is covered by a nilpotent totally symmetric dessin. Totally symmetric dessins with nilpotent automorphism groups of class one and two have been classified in \cite{HNW2014, HNW2014-3}. The main result of the present paper is a classification of totally symmetric dessins with nilpotent automorphism groups of nilpotent class three, see Theorem~\ref{MAIN}.

\section{Algebraic theory of dessins}
In this section we outline the algebraic theory of regular dessin and dessin operations. For more details the reader is referred to \cite{JP2010} and \cite{Jones2013}.

Each dessin $D$ on an orientable surface $S$ determines a two-generator transitive permutation group $\Mon(D)=\langle \rho,\lambda\rangle$ on the edge set $E$ of $D$: The global orientation of $S$ induces two permutations $\rho$ and $\lambda$ which cyclically permute the edges incidenct with each black and each white vertice of $D$; since the underlying graph of $D$ is connected, the group $\Mon(D)=\langle \rho,\lambda\rangle$ is transitive. Let 
\[
F_2=\langle X,Y\mid -\rangle,
\]
i.e., $F_2$ is the free group of rank two. Then we have a transitive permutation representation of $F_2$ defined by 
\[
F_2\to\Mon(D), X\mapsto \rho,Y\mapsto \lambda.
\]
 In this representation the black and the white vertices correspond  to the orbits of the subgroups $\langle X\rangle$ and $\langle Y\rangle$, respectively, and the faces correspond to the oribits of $\langle XY\rangle$, with incidence given by non-empty intersection. 

The stabilizer $N:=\mathrm{Stab}_{F_2}(e)$  in $F_2$ of an edge $e$ is a subgroup of finite index in $F_2$. This subgroup is uniquely determined up to conjugacy. It will be called the \textit{dessin subgroup} associated with $D$.  Moreover, an automorphism $\sigma$ of $F_2$ sends $N$ to $N^{\sigma}$, and hence transforms $D$ to a dessin $D^{\sigma}$. In particular if $\sigma$ is an inner automorphism of $F_2$, then $N$ is conjugate to $N^{\sigma}$, and hence $D$ is isomorphic to $D^{\sigma}$. It follows that the outer automorphism group 
\[
\Omega:=\Out(F_2)=\Aut(F_2)/\Inn(F_2)
\]
of $F_2$ acts as the \textit{group of dessin operations} on the isomorphism classes of dessins. It is well known that $\Aut(F_2)$ is genearated by the elementary Nielsen transformations of the form
\begin{align*}
&\tau: X\to Y, Y\to X,  &&\pi_1:X\to X, Y\to Y^{-1}, &&\pi:X\to X^{-1}, Y\to Y,\\
&\zeta:X\to XY, Y\to Y, &&\eta:X\to X, Y\to YX. &&
\end{align*} 
This is not a minimal set of generators. For instance we have
\begin{proposition}
$\Aut(F_2)=\langle\tau,\pi,\zeta\rangle$ and $\Omega=\langle \omega_{\tau},\omega_{\pi},\omega_{\zeta}\rangle$ where $\omega_{\tau},\omega_{\pi},\omega_{\zeta}$ are dessin operations induced by the automorphisms $\tau,\pi,\zeta$ of $F_2$.
\end{proposition}
 \begin{proof}
Note that $\pi_1=\tau\pi\tau$ and $\eta=\tau\zeta\tau$, we have $\Aut(F_2)=\langle \tau,\pi,\zeta\rangle$, and hence $\Omega=\langle \omega_{\tau},\omega_{\pi},\omega_{\zeta}\rangle$.
\end{proof}

The operation $\omega_\tau$ on dessins transposes the black and white vertices while preserves the faces and orientation. The operation $\omega_{\pi}$ reverses the orientation around the black vertices while preserves it around the white vertices. It is sometimes called the Petrie operation, since it transposes the faces with the Petrie polygons (the zig-zag paths). Finally, let  $\iota=\pi\pi_1$, then the operation $\omega_{\iota}$ transforms a dessin into its mirror image by reversing the orientation around the black and white vertices.

The most important dessins are regular dessins, in which both $\Mon(D)$ and $\Aut(D)$ are regular permutation groups on $E$, and the associated dessin subgroups $N$ are normal in $F_2$. In particular we have
\[
\Mon(D)\cong\Aut(D)\cong F_2/N.
\]
In a regular dessin $D$ one may identify $D$ with the underlying set of $G:=\Aut(D)$, and the groups $\Aut(D)$ and $\Mon(D)$ with the left and right regular representations of $G$, respectively. So each regular dessin $D$ gives rise to an \textit{algebraic dessin} -- a triple $(G,x,y)$ where $x$ and $y$ generate the stabilizers of a pair of  adjacent (black and white) vertices in $G$, and hence $G=\langle x,y\rangle$. 

 Moreover, for an automorphism $\sigma\in\Aut(F_2)\backslash\Inn(F_2)$, if the normal subgroup $N$ in $F_2$ is $\sigma$-invariant, i.e., $N=N^{\sigma}$, then the assoicated regular dessin $D$ is invariant under the dessin operation $\omega_\sigma$, i.e., $D\cong D^{\sigma}$. In this case we say that the regular dessin $D$ posseses the \textit{external symmetry} $\sigma$.
More specifically, a regular dessin $D$ will be called \textit{symmetric} (resp. \textit{self-Petrie-dual}, \textit{reflexible}) if it is invariant under the operation $\omega_{\tau}$ (resp. $\omega_{\pi}$, $\omega_{\iota}$). A regular dessin which is invariant under all dessin operations will be called \textit{totally symmetric}. It follows that a regular dessin is totally symmetric if and only if the associated dessin subgroup $N$ is characteristic in $F_2$, in which case we have a natural homomorphism $\Aut(F_2)\to\Aut(G)$ where $G\cong F_2/N$. In particular each of the following assignments extends to an automorphism of $G$ (By abuse of notation we have retained the Greek letters used above):
\begin{equation}\label{MinG}
\begin{aligned}
\tau:& x\mapsto y, \quad y\mapsto x;   \\
 \pi:& x\mapsto x^{-1},\quad y\mapsto y;\\
  \zeta:&x\mapsto xy,\quad y\mapsto y.
\end{aligned}
\end{equation}
Define $ \iota=\pi\tau\pi\tau$ and $\nu=\iota\tau\zeta\tau\iota$. Then
\[
\iota: x\mapsto x^{-1},\quad y\mapsto y^{-1}\quad\text{and}\quad  \nu:x\mapsto x,\quad y\mapsto xy.
\]

\begin{lemma}\label{INTERSECTION}
Let $(G,x,y)$ be a totally symmetric dessin, then $\langle y\rangle\cap G'=\langle y\rangle\cap \langle x, G'\rangle.$
\end{lemma}
\begin{proof}
It suffices to prove that $\langle y\rangle\cap \langle x, G'\rangle\subseteq G'$. Assume that $\langle y\rangle\cap \langle x, G'\rangle=\langle y^{p^m}\rangle$ where $m$ is a nonnegative integer. Then 
\begin{align}\label{EQNN1}
y^{p^m}=x^{rp^n}g_1,
\end{align}
where $r$ and $n$ are nonnegative integers, $\gcd(r,p)=1$, and $g_1\in G'$. Applying $\tau$ to \eqref{EQNN1} we have
\begin{align}\label{EQNN2}
x^{p^m}=\tau(y^{p^m})=\tau(x^{rp^n}g_1)=\tau(x^{rp^n})\tau(g_1)=y^{rp^n}\tau(g_1).
\end{align}
So we have $y^{rp^n}=x^{p^m}\tau(g_1)^{-1}\in \langle x, G'\rangle$. By the minimality of $m$ we have $m\leq n$.

If $m=n$, then applying $\nu=\iota\tau\zeta\tau\iota: x\mapsto x, y\mapsto xy$ to \eqref{EQNN1} we have
\begin{align*}
(xy)^{p^m}=\nu(y^{p^m})=\nu(x^{rp^n}h)=x^{rp^m}\nu(g_1).
\end{align*}
By  Hall-Petrescu formula~\cite[Theorem~A.1.3]{BJ2008} $(xy)^{p^m}=x^{p^m}y^{p^m}g_2$ where $g_2\in G'$. So $x^{p^m}y^{p^m}g_2=x^{rp^m}\nu(g_1)$, and hence 
$x^{(1-r)p^m}y^{p^m}=\nu(g_1)g_2^{-1}.$ It follows that
\[
x^{p^m}=x^{(1-r)p^m}x^{rp^m}\stackrel{\eqref{EQNN1} }=x^{(1-r)p^m}y^{p^m}g_1^{-1}=\nu(g_1)g_2^{-1}g_1^{-1}\in G'.
\]
Therefore $y^{p^m}=\tau(x^{p^m})=\tau(\nu(g_1)g_2^{-1}g_1^{-1})\in G'$.

On the other hand, if $m<n$, then by \eqref{EQNN1} and \eqref{EQNN2} we have
\[
y^{p^m}\stackrel{ \eqref{EQNN1}}=x^{rp^n}g_1=(x^{p^m})^{rp^{n-m}}g_1 \stackrel{\eqref{EQNN2}}=(y^{rp^n}\tau(g_1))^{rp^{n-m}}g_1=y^{r^2p^{2n-m}}g_3,
\]
where $g_3\in G'$. So $y^{p^m(1-r^2p^{2n-2m})}=g_3\in G'.$
Since $\gcd(r,p)=1$ and $n>m$, we have $\gcd(p, 1-r^2p^{2n-2m})=1$. Consequently $y^{p^m}\in G'$.
\end{proof}

\begin{proposition}\cite{Jones2013} \label{FUNDM}
Let $G$ be a finite two-generator group, then
\begin{itemize}
\item[\rm(i)] the set $\mathcal{R}(G)$ of isomorphism classes of regular dessins $D$ with $\Aut(D)\cong G$ corresponds bijectively to the set $\mathcal{N}(G)$ of normal subgroups $N$ of $F_2$ such that $F_2/N\cong G$. 
\item[\rm(ii)] The isomorphism classes of regular dessins $D$ with $\Aut(D)\cong G$ are in one-to-one correspondence with the orbits of the action of $\Aut(G)$ on the generating pairs of $G$. 
\end{itemize}
\end{proposition}
Note that the number $|\mathcal{R}(G)|=|\mathcal{N}(G)|$ is finite. Define
\begin{align*}
K(G)=\bigcap_{N\in \mathcal{N}(G)} N.
\end{align*}
Then $K(G)$, being the intersection of finitely many normal subgroups of finite index in $F_2$,   is normal of finite index in $F_2$. Let $U(G)$ be the regular dessin associated with $K(G)$, and $\bar G=F_2/K(G)$. Then $U(G)$ is the smallest regular dessin which covers all regular dessins in $\mathcal{R}(G)$. The construction of $U(G)$ and $\bar G$ implies that $U(G)$ is the \textit{unique} regular dessin with an automorphism group isomorphic to $\bar G$. The uniqueness implies that $U(G)$ is totally symmetric, and is defined over the field of rational numbers. In particular, if $G$ is nilpotent of class $c$, then so is $\bar G$~\cite{Jones2013}.

 \begin{example}\label{ABELIAN}\cite[Example 5.1]{Jones2013}
Let $m,n$ be positive integers, and $m\mid n$. If  $G\cong C_n\times C_m$,  the direct product of two cyclic groups of orders $n$ and $m$, then $\bar G\cong C_n\times C_n$,  and $U(G)$ is the $n$th degree Fermat dessin, corresponding to the standard embedding of $K_{n,n}$~\cite{Jones2010} into the Fermat curve
\[
x_0^n+x_1^n+x_2^n=0.
\]
 in $\mathbb{P}^2(\mathbb{C})$.
 \end{example}

For a regular dessin $D=(G,x,y)$, the type of $D$ is the triple $(l,m,n)$ where 
\[
l=o(x),\quad m=o(y),\quad n=o(xy),
\]
the orders of the elements $x,y$ and $xy$. It has genus $g$ determined by the Euler-Poincar\'e formula:
\[
2-2g=|G|(\frac{1}{l}+\frac{1}{m}+\frac{1}{n}-1).
\]

\section{Nilpotent groups}
In this section, we summarise some preliminary results on nilpotent groups to be used later.

Let $G$ be a group, the commutators of order $n$ are defined recursively by the rule
 \[
[x_1,x_2]=x_1^{-1}x_2^{-1}x_1x_2, \quad [x_1,x_2,\ldots, x_n]=[[x_1,x_2,\ldots,x_{n-1}],x_n], n>2
 \]
 where $x_i\in G$.
For subgroups $H,K$ of $G$, the notation $[H,K]$ will mean the group generated by all $[x,y]$ with $x\in H$ and $y\in K$. In particular, $[G,G]$ is the commutator subgroup of $G$. We define recursively the $n$th derived subgroup of $G$ by the rule
\[G'=G^{(1)}=[G,G],\quad G^{''}=G^{(2)}=[G',G'],\quad G^{(n)}=[G^{(n-1)},G^{(n-1)}], n\geq 3.
\]
A group is solvable if the sequence
\[
G\geq G'\geq G^{''}\geq\cdots
\]
terminates at the identity group in a finite number of steps, say $G^{(d)}=1$. The number $d$ is the derived length of $G$.  In particular, $G$ is called \textit{metabelian} if $G^{''}=1$.
 Moreover, the \textit{lower central series} of $G$ is
 \[
 G=G_1\geq G_2\geq G_3\geq\cdots
 \]
 where $G_2=[G,G]$ and $G_i=[G_{i-1},G]$ $(i\geq 3)$, and the \textit{upper central series} of $G$ is
 \[
 1=Z_0(G)\leq Z_1(G)\leq Z_2(G)\leq\cdots
 \]
 where $Z_1(G)=Z(G)$, the center of $G$, and $Z_i(G)/Z_{i-1}(G)=Z(G/Z_{i-1}(G))$ $(i\geq 2)$. A group is \textit{nilpotent} if its lower central series terminates at the identity group in a finite number of steps.
 This is equivalent to that its upper central series contains $G$ in a finite steps~\cite{Huppert1967}. If $G$ is a nilpotent, then its upper central series and lower central series have finite length and both have the same length $c$. The number $c$ is  called the class of $G$. In general if $G$ is nilpotent of class $c$, then for each $j$ $(0\leq j\leq c)$, $G_{c+1-j}\leq Z_j(G)$~\cite[Lemma 1.7.9]{XQ2010}. So if $c=3$ and $j=0$, then $G_4\leq Z_0(G)=1$. Since $G^{(i)}\leq G_{2^i}$~\cite[Theorem III.2.12]{Huppert1967}, we have $G''=1$. That is, $G$ is a metabelian group.

 \begin{lemma}\label{META2}\cite[Proposition 2.1.5]{XQ2010}
Let $G$ be a metabelian group, $x_i\in G$ $(i=1,2,\ldots,n)$. If $z\in G'$, then  \[[z,x_1,x_2,\ldots,x_n]=[z,x_{\alpha(1)},x_{\alpha(2)},\ldots, x_{\alpha(n)}].
 \]
where $\alpha$ is an arbitrary permutation of $\{1,2,\ldots, n\}$.
\end{lemma}
By Lemma~\ref{META2}, for brevity, we may define
\[[ix,jy]=[x,y, x\underbrace{\ldots}_{i-1}x,y\underbrace{\ldots}_{j-1}y].\]
 where $i$ and $j$ are positive integers.

\begin{lemma}\label{META3}\cite[Proposition 2.1.6]{XQ2010}
Let $G$ be a metabelian group generated by $x$ and $y$. Then for any $s\geq2$,
\[
G_s=\la [ix,(s-i)y], G_{s+1}\mid i=1,2,\cdots,s-1\ra.
\]\end{lemma}

\begin{lemma}\label{META4}\cite[Proposition 2.1.7]{XQ2010}
Let $G$ be a metabelian group, $x,y\in G$. Then for any positive integers $m$ and $n$
\begin{align}
[x^m,y^n]=&\prod_{i=1}^m\prod_{j=1}^n[ix,jy]^{{m\choose i}{n\choose j}}.
\end{align}

\end{lemma}

 \begin{lemma}\label{META5}\cite[Proposition 2.1.8]{XQ2010}
Let $G$ be a metabelian group, $x,y\in G$. Then for any integer $m\geq2$
\begin{align}
(xy^{-1})^m=&x^m\Big(\prod_{i+j\leq m}[ix,jy]^{m\choose{i+j}}\Big)y^{-m}.
\end{align}
\end{lemma}

 \begin{corollary}\label{META6}
Let $G$ be a nilpotent group of class 3, $x,y\in G$. Then for any nonnegative integers $m$ and $n$,
\begin{align*}
&[x^m,y^n]=[x,y]^{mn}[x,y,x]^{n{m\choose 2}}[x,y,y]^{m{n\choose 2}}, \\
&(xy^{-1})^m=x^my^{-m}[x,y]^{m\choose2}[x,y,x]^{m\choose3}[x,y,y]^{{m\choose3}-m{m\choose2}},\\
&(xy)^m=x^my^m[x,y]^{-{m\choose2}}[x,y,x]^{-{m\choose3}}[x,y,y]^{{m\choose3}+(1-m){m\choose2}},
\end{align*}
where ${m\choose j}=0$ if $m<j$.
\end{corollary}
\begin{proof}These identities can be easily derived from Lemma~\ref{META4} and \ref{META5}.
\end{proof}
\begin{corollary}\label{META7}
Let $G$ be a nilpotent group of class 3, $x,y\in G$. Then
for any nonnegative integers $i,j,k,l$ and $m$,
\begin{align*}
&(x^iy^j)^m=x^{mi}y^{mi}[x,y]^{-ij{m\choose2}}[x,y,x]^{-j{i\choose2}{m\choose2}-i^2j{m\choose3}}[x,y,y]^{-i{j\choose2}{m\choose2}+ij^2\big({m\choose3}+(1-m){m\choose2}\big)},\\
&[x^iy^j,x^ky^l]=[x,y]^{il-jk}[x,y,x]^{l{i\choose2}-j{k\choose2}}[x,y,y]^{ijl-kjl+i{l\choose2}-k{j\choose2}}.
\end{align*}
\end{corollary}
\begin{proof} By Corollary~\ref{META6}, we have
\begin{align*}
(x^iy^j)^m=&x^{mi}y^{mi}[x^i,y^j]^{-{m\choose 2}}[x^i,y^j,x^i]^{-{m\choose3}}[x^i,y^j,y^j]^{{m\choose3}+(1-m){m\choose2}}\\
=&x^{mi}y^{mi}\big([x,y]^{ij}[x,y,x]^{j{i\choose2}}[x,y,y]^{i{j\choose2}}\big)^{-{m\choose 2}}\\
&\big[[x,y]^{ij}[x,y,x]^{j{i\choose2}}[x,y,y]^{i{j\choose2}},x^i\big]^{-{m\choose3}}\\
&\big[[x,y]^{ij}[x,y,x]^{j{i\choose2}}[x,y,y]^{i{j\choose2}},y^j\big]^{{m\choose3}+(1-m){m\choose2}}\\
=&x^{mi}y^{mi}[x,y]^{-ij{m\choose2}}[x,y,x]^{-j{i\choose2}{m\choose2}}[x,y,y]^{-i{j\choose2}{m\choose2}}\\
&[x,y,x]^{-i^2j{m\choose3}}[x,y,y]^{ij^2\big({m\choose3}+(1-m){m\choose2}\big)}\\
=&x^{mi}y^{mi}[x,y]^{-ij{m\choose2}}[x,y,x]^{-j{i\choose2}{m\choose2}-i^2j{m\choose3}}\\
&[x,y,y]^{-i{j\choose2}{m\choose2}+ij^2\big({m\choose3}+(1-m){m\choose2}\big)}.
\end{align*}

\begin{align*}
[x^iy^j,x^ky^l]=&[x^i,x^ky^l]^{y^j}[y^j,x^ky^l]\\
=&[x^i,y^l]^{y^j}[y^j,x^k]^{y^l}\\
=&\big([x,y]^{il}[x,y,x]^{l{i\choose2}}[x,y,y]^{i{l\choose2}}\big)^{y^j}\big([y,x]^{jk}[y,x,y]^{k{j\choose2}}[y,x,x]^{j{k\choose2}}\big)^{y^l}\\
=&([x,y]^{y^j})^{il}([y,x]^{y^l})^{jk}[x,y,x]^{l{i\choose2}-j{k\choose2}}[x,y,y]^{i{l\choose2}-k{j\choose2}}\\
=&([x,y][x,y,y^j])^{il}([y,x][y,x,y^l])^{jk}[x,y,x]^{l{i\choose2}-j{k\choose2}}[x,y,y]^{i{l\choose2}-k{j\choose2}}\\
=&[x,y]^{il-jk}[x,y,x]^{l{i\choose2}-j{k\choose2}}[x,y,y]^{ijl-kjl+i{l\choose2}-k{j\choose2}}.
\end{align*}
\end{proof}

 \begin{lemma}\cite[Chapter III, Theorem 3.15]{Huppert1967}\label{BURNSIDE}
Let $G$ be a $p$-group and $\Phi(G)$ the Frattini subgroup of $G$. If $|G/\Phi(G)|=p^d,$ then every minimal generating set of $G$ contains exactly $d$ elements. In particular, $G=\la x_i\mid i=1,2,\ldots, d\ra$ if and only if $G/\Phi(G)=\la x_i\Phi(G)\mid i=1,2,\ldots, d\ra$.
\end{lemma}


\section{Nilpotent totally symmetric dessins of class three}
The rest of the paper is devoted to the classification of totally symmetric dessins $D=(G,x,y)$ with a nilpotent automorphism group $G$ of class three. Since each nilpotent group is a direct product of its Sylow subgroups, it suffices to consider the case when  $G$ is a $p$-group of class three. For simplicity such a dessin will be called a \textit{totally symmetric $p$-dessin  of class three}. In this section we give conventions and prove several lemmas. 

Let $D=(G,x,y)$ be a totally symmetric $p$-dessin of class three.
Define
  \[z=[x,y],\quad u=[z,x]=[x,y,x]\quad \text{and}\quad v=[z,y]=[x,y,y]\]
Since $D$ is totally symmetric, by Eq.~\eqref{MinG} each of the following assignments extends to an automorphism of $G\cong F_2/N$:
\begin{align*}
\tau:& x\mapsto y, \quad y\mapsto x;   \\
 \pi:& x\mapsto x^{-1},\quad y\mapsto y;\\
  \zeta:&x\mapsto xy,\quad y\mapsto y;\\
\iota:& x\mapsto x^{-1}, \quad y\mapsto y^{-1};\\
\xi:&x\mapsto xy^{-1}, y\mapsto y.\\
\end{align*}
where $\iota=\pi\tau\pi\tau$ and $\xi=\zeta^{-1}.$
\begin{lemma}\label{TRANS} Let $(G,x,y)$ be a totally symmetric $p$-dessin of class three, and denote $z=[x,y]$, $u=[z,x]$ and $v=[z,y].$ Then under the automorphisms $\tau$,  $\pi$, $\zeta$, $\iota$ and $\xi$ of $G$ the images of $z$, $u$ and $v$ are summarised in the following table.
\[
\begin{array}{|c|lllll|}
\hline
& x      & y     & z           & u              &v\\
\hline
\tau       & y      & x     &z^{-1}       & v^{-1}        &u^{-1}\\
\pi     & x^{-1} & y     &z^{-1}u      & u             &v^{-1}\\
\zeta       & xy     & y     &zv           & uv            &v     \\
\iota      &x^{-1}  & y^{-1}&zu^{-1}v^{-1}&u^{-1}         &v^{-1}\\
\xi      &xy^{-1} & y     &zv^{-1}      &uv^{-1}        &v\\
\hline
\end{array}
\]
\end{lemma}
\begin{proof}
Since $x^{\tau}=y$ and $y^{\tau}=x$, we have
\begin{align*}
&z^{\tau}=[x,y]^{\tau}=[x^{\tau},y^{\tau}]=[y,x]=z^{-1},\\
&u^{\tau}=[z,x]^{\tau}=[z^{\tau},x^{\tau}]=[z^{-1},y]=[z,y]^{-1}=v^{-1},\\
&v^{\tau}=[z,y]^{\tau}=[z^{\tau},y^{\tau}]=[z^{-1},x]=[z,x]^{-1}=u^{-1}.
\end{align*}
The proof for other cases is similar and is left to the reader.
\end{proof}

Now assume that
  \[
  o(x)=p^a,\quad o(z)=p^b\quad\text{and}\quad o(u)=p^c,
  \]
where $a,b,c$ are nonnegative integers.
 By Lemma~\ref{META3}, $G_3=\la u, v\ra$ and $G_2=G'=\la z, G_3\ra$. Define $H=\la G',x\ra$. Then $G=\la H, y\ra$ and $H\normal G$. Assume that 
\[
\la x\ra\cap G'=\la x^{p^d}\ra\quad \text{and}\quad \la y\ra\cap H=\la y^{p^e}\ra
\]
 where $d,e$ are nonnegative integers. In other words, $d$ and $e$ are the smallest nonnegative integers $i$ and $j$ such that $x^{p^i}\in G'$ and $y^{p^j}\in H$. Note that 
\[
\langle y\rangle\cap G'=\tau(\langle x\rangle\cap G')=\tau(\langle x^{p^d}\rangle)=\langle \tau( x^{p^d})\rangle=\langle y^{p^d}\rangle
\]
By Lemma~\ref{INTERSECTION} $\langle y\rangle\cap H=\langle y\rangle\cap G'$,  we have $d=e$. Assume that 
\begin{align}
&x^{p^d}=z^lu^mv^n, \label{IDEN1}
\end{align}
where $l,m,n$ are nonnegative integers such that $ 0\leq l< p^b$ and $0\leq m,n<p^c$.
Then 
\begin{align}\label{IDEN2}
y^{p^d}=\tau(x^{p^d})=\tau(z^lu^mv^n)=z^{-l}v^{-m}u^{-n}.
\end{align} 

\begin{lemma}\label{DISJOINT} Let $(G,x,y)$ be a totally symmetric $p$-dessin of class three. Then with the above notation the following statements hold true:
\begin{enumerate}
\item[\rm(i)]$o(y)=o(x)$ and $o(u)=o(v)$;
\item[\rm(ii)]$\la u\ra\cap\la v\ra=\la z\ra\cap G_3=\la x\ra\cap\la v\ra=\la y\ra\cap\la u\ra=1;$
\item[\rm(iii)] $1\leq c\leq b\leq d\leq a;$
\item[\rm(iv)] $n=0$.
\end{enumerate}
\end{lemma}
\begin{proof} (i) By Lemma~\ref{TRANS}, the automorphism $\tau$ of $G$ transposes $x$ and $y$, and $u$ and $v^{-1}$. So $o(y)=o(\tau(y))=o(x)$ and $o(u)=o(\tau(u))=o(v^{-1})=o(v)$.

(ii) Assume that $\la u\ra\cap\la v\ra=\la v^h\ra$. Since $o(u)=o(v)$, we have $u^h=v^{hf}$ for some $f$ coprime to $p$. Applying the automorphism $\zeta$ to the relation we obtain that $\zeta(u^h)=\zeta(v^{hf})$. By Lemma~\ref{TRANS} $\zeta(u^h)=\zeta(u)^h=(uv)^h=u^hv^h$ and $\zeta(v^{hf})=v^{hf}=u^h$. It follows that $u^hv^h=u^h$, and hence $v^h=1$. This proves that $\la u\ra\cap\la v\ra=1$.

Moreover, assume that $\la z\ra\cap\la u,v\ra=\la z^k\ra$. Then $z^k=u^iv^j$ for some integers $i$ and $j$. Applying $\tau$ to the relation we get $\tau(z^k)=\tau(u^iv^j)$. By Lemma~\ref{TRANS} we have $\tau(z^k)=\tau(z)^k=z^{-k}=u^{-i}v^{-j}$ and $\tau(u^iv^j)=\tau(u)^i\tau(v)^j=u^{-j}v^{-i}.$ It follows that $u^{i-j}=v^{i-j}$. Recall that $o(u)=o(v)=p^c$ and $\la u\ra\cap\la v\ra=1$. So we have $i\equiv j\pmod{p^c}$, and hence 
\[
z^k=u^iv^i.
\]
 Applying $\zeta$ to the relation we get $\zeta(z^k)=\zeta(u^iv^i)$. By Lemma~\ref{TRANS} we have $\zeta(z^k)=\zeta(z)^k=(zv)^k=z^kv^k$ and $\zeta(u^iv^i)=\zeta(u)^i\zeta(v)^i=u^iv^{2i}$. So we obtain that $z^k=u^iv^{2i-k}$. Comparing this with previous relation we get $v^{i-k}=1$. So $i\equiv k\pmod{p^c}$.  On the other hand, applying $\iota$ to the relation $z^k=u^iv^i$ we deduce that $z^k=u^{k-i}v^{k-i}$. Since $i\equiv k \pmod{p^c}$ and $o(u)=o(v)=p^c$, we obtain that $z^k=1$. Consequently $\la z\ra\cap\la u,v\ra=1$.

Finally, assume that $\la x\ra\cap\la v\ra=\la x^i\ra$ where $i\geq0$. Then $x^i=v^j$ for some integer $j$. Applying $\tau$ to the relation we have $y^i=\tau(x^i)=\tau(v^j)=u^{-j}$. By Lemma~\ref{META6} and Lemma~\ref{TRANS} we have
\[
v^j=\xi(v^j)=\xi(x^i)=(xy^{-1})^i=x^iy^{-i}z^{i\choose 2}u^{i\choose 3}v^{{i\choose 3}-i{i\choose 2}}=z^{i\choose2}u^{j+{i\choose3}}v^{j+{i\choose3}-i{i\choose2}},
\]
which is equivalent to
\[
z^{-{i\choose 2}}=u^{j+{i\choose3}}v^{{i\choose3}-i{i\choose2}}.
\]
Recall that $\la u\ra\cap\la v\ra=\la z\ra\cap\la u,v\ra=1$, so we have
\begin{align}
&{i\choose 2}\equiv0\pmod{p^b},\label{DIS1}\\
&{i\choose3}\equiv-j\pmod{p^c},\label{DIS2}\\
&{i\choose3}\equiv i{i\choose2}\pmod{p^c}.\label{DIS3}
\end{align}
By induction we deduce that $u^{p^b}=[z,x]^{p^b}=[z^{p^b},x]=1$. Recall that $o(u)=p^c$. So we get $c\leq b$. So by \eqref{DIS1} we have ${i\choose2}\equiv0\pmod{p^c}$. Hence  \eqref{DIS3} is reduced to ${i\choose3}\equiv0\pmod{p^c}$. Consequently, \eqref{DIS2} is reduced to $j\equiv0\pmod{p^c}$. Therefore by the assumption $\la x\ra\cap\la v\ra=1$. Applying $\tau$ we also have $\la y\ra\cap\la u\ra=\tau(\la x\ra\cap\la v\ra)=1$.

(iii) Since $G$ has class 3, $G_4=1$ and $G_3>1$. Recall that $G_3=\la u,v\ra$ and $o(u)=o(v)=p^c$. So we have $c\geq 1$. In the proof of (ii) we have already shown $c\leq b$. Moreover, by Corollary~\ref{META6} we have $[x,y^{p^d}]=z^{p^d}v^{{p^d}\choose 2}$. By \eqref{IDEN2} we also have 
\[
[x,y^{p^d}]=[x,z^{-l}v^{-m}u^{-n}]=[x,z^{-l}]=[x,z]^{-l}=u^{l}.
\]
It follows that $z^{p^d}v^{{p^d}\choose 2}=u^l$, which implies that $z^{p^d}=u^{l}v^{-{p^d\choose2}}$. By Lemma~\ref{DISJOINT}, $\la z\ra\cap\la u,v\ra=\la u\ra\cap\la v\ra=1$. So the above relation implies that $z^{p^d}=1$. Consequently $b\leq d$.

(iv) Applying $\pi$ to \eqref{IDEN1} and by Lemma~\ref{TRANS} we have
\begin{align*}
x^{-p^d}=\pi(x^{p^d})&=\pi(z^lu^mv^n)=z^{-l}u^{l+m}v^{-n}
\end{align*}
 Combining this with \eqref{IDEN1} we deduce that $z^{-l}u^{l+m}v^{-n}=z^{-l}u^{-m}v^{-n}$, or equivalently, $u^{l+2m}=1$. Hence
 \begin{align}\label{EQN1}
 l+2m\equiv0\pmod{p^c}.
 \end{align}

On the other hand, applying $\xi$ to \eqref{IDEN1} we have
 \begin{align*}
 z^lu^mv^{-l-m+n}&=\xi(z^lu^mv^n)=\xi(x^{p^d})=(xy^{-1})^{p^d}\\
 &=x^{p^d}y^{-p^d}z^{p^d\choose 2}u^{p^d\choose 3}v^{{p^d\choose 3}-p^d{p^d\choose 2}}\\
 &=(z^lu^mv^n)(z^{l}u^{n}v^{m})z^{p^d\choose 2}u^{p^d\choose 3}v^{{p^d\choose 3}-p^d{p^d\choose 2}}\\
 &=z^{2l+{p^d\choose2}}u^{m+n+{p^d\choose3}}v^{m+n+{p^d\choose 3}-p^d{p^d\choose 2}}.
 \end{align*}
Hence \[z^{-l-{p^d\choose2}}=u^{n+{p^d\choose3}}v^{l+2m+{p^d\choose3}-p^d{p^d\choose2}}.\]
    By (ii), $\la z\ra\cap\la u,v\ra=\la u\ra\cap\la v\ra=1$. It follows that
 \begin{align}
 &l+{p^d\choose2}\equiv0\pmod{p^b},\label{EQN2}\\
 &n+{p^d\choose3}\equiv0\pmod{p^c},\label{EQN3}\\
 &l+2m+{p^d\choose3}-p^d{p^d\choose2}\equiv0\pmod{p^c}.\label{EQN4}
 \end{align}
 By (iii), $c\leq d$. so $p^d{p^d\choose2}\equiv0\pmod{p^c}$. Combining this with \eqref{EQN1} the congruence \eqref{EQN4} is reduced to
 \begin{align}\label{OEQN}
  {p^d\choose 3}\equiv0\pmod{p^c}.
  \end{align}
Hence \eqref{EQN3} is reduced to $n\equiv0\pmod{p^c}.$
\end{proof}

\begin{corollary}\label{ALLCASE}
The values of numberical parameters $l,m,n$ and $a,b,c,d$ defined above bolong to one of the following cases:
\begin{enumerate}
\item[\rm(A)] $p$ is odd,
\begin{enumerate}
\item[\rm(i)]$p>3$ and  $l=m=n=0$ and $1\leq c\leq b\leq d=a$;
    \item[\rm(ii)]$p=3$ and $l=m=n=0$ and either $1\leq c< b=d=a$ or $1\leq c\leq b<d=a$;
\end{enumerate}
\item[\rm(B)] $p=2$,
\begin{enumerate}
\item[\rm (iii)]$l=m=n=0$ and $1\leq c\leq b<d=a$,
    \item[\rm (iv)] $l=n=0$ and $m=2^{c-1}$ and $1\leq c\leq b\leq d-1=a-2$,
     \item[\rm (v)] $l=2^{b-1}$ and $m=n=0$ and $1\leq c\leq b-1$ and $b=d=a-1$,
      \item[\rm (vi)]$l=2^{b-1}$, $m=q=2^{c-1}$, $n=0$ and $1\leq c\leq b-1$ and $b=d=a-1$.
\end{enumerate}
\end{enumerate}
\end{corollary}
\begin{proof}
By Lemma~\ref{DISJOINT}(iv), $n=0$, so Eqs.~\eqref{IDEN1} and \eqref{IDEN2} are reduced to
 \begin{align}\label{IDEN4}
 x^{p^d}=z^lu^m\quad\text{and}\quad y^{p^d}=z^{-l}v^{-m}.
 \end{align}
 Applying $\iota$ to the first relation we have \[z^lu^{-l-m}v^{-l}= \iota(z^lu^m)=\iota(x^{p^d})=x^{-p^d}=z^{-l}u^{-m},\]
 so we obtain $z^{2l}=u^{l}v^l$. Since $\la z\ra\cap\la u,v\ra=\la u\ra\cap\la v\ra=1$ (see Lemma~\ref{DISJOINT}(ii)), we get
 \begin{align}
  2l&\equiv0\pmod{p^b},\label{BEQN4}\\
  l&\equiv0\pmod{p^c}.\label{BEQN5}
   \end{align}Combining \eqref{BEQN5} with \eqref{EQN1} we obtain that \begin{align}
   2m\equiv0\pmod{p^c}.\label{BEQN6}
   \end{align}

\begin{case}[A]If $p$ is odd, by \eqref{BEQN4} and \eqref{BEQN6}  we have $l=m=0$. Hence by the first relation in \eqref{IDEN4} we have $x^{p^d}=1$. Therefore $d=a$. Combining this with Lemma~\ref{DISJOINT}(iii) we thus obtain $1\leq c\leq b\leq d=a$. Eq.~\eqref{OEQN} motivates us to distinguish two subcases.
\begin{subcase}[i]$p>3$.\par
In this case, \eqref{OEQN} is automatically satisfied.
\end{subcase}

\begin{subcase}[ii]$p=3$.\par
Since $3^{d-1}||{p^d\choose3}$, by \eqref{OEQN}, we have $c<d$. So either $1\leq c<b=d=a$ or $1\leq c\leq b<d=a$.
\end{subcase}
\end{case}
\begin{case}[B] If $p=2$, then by \eqref{BEQN4} and \eqref{BEQN6} we have one of the following four subcases: (iii) $l=m=0$; (iv) $l=0$ and $m=2^{c-1}$; (v) $l=2^{b-1}$ and $m=0$; (vi) $l=2^{b-1}$ and $m=2^{c-1}$. 
\begin{subcase}[iii] $l=m=0$.\par
 In this case,  Eq.~\eqref{EQN2} is reduced to ${2^d\choose 2}\equiv0\pmod{2^b}$, so $b\leq d-1$; The first relation in \eqref{IDEN4} is reduced to $x^{p^d}=1$, so $d=a$.  Combining these the inequality in Lemma~\ref{DISJOINT}(iii), we get $1\leq c\leq b\leq d-1=a-1$.
 \end{subcase}

\begin{subcase}[iv] $l=0$ and $m=2^{c-1}$.\par
In this case,  Eq.~\eqref{EQN2} is also reduced to ${2^d\choose 2}\equiv0\pmod{2^b}$, so $b\leq d-1$;  the first relation in \eqref{IDEN4} is reduced to $x^{p^d}=u^{2^{c-1}}$, so $d=a-1$. Combining these with the inequalities in Lemma~\ref{DISJOINT}(iii)-(iv) we thus obtain $1\leq c\leq b\leq d-1=a-2$.
\end{subcase}

\begin{subcase}[v] $l=2^{b-1}$ and $m=0$.\par
In this case,  Eq.~\eqref{EQN2} is reduced to $2^{b-1}+{2^d\choose 2}\equiv0\pmod{2^b}$. By \eqref{BEQN5} we have $2^{b-1}\equiv0\pmod{2^c}$. So $c\leq b-1$ and $b=d$. Note that the first relation in \eqref{IDEN4} is reduced to $x^{p^d}=z^{2^{b-1}}$, so $d=a-1$. Combining these with the inequalities in Lemma~\ref{DISJOINT}(iii)-(iv) we thus obtain $1\leq c\leq b= d=a-1$
\end{subcase}

\begin{subcase}[vi] $l=2^{b-1}$ and $m=2^{c-1}$.\par
In this case,  Eq.~\eqref{EQN2} is reduced to $2^{b-1}+{2^d\choose 2}\equiv0\pmod{2^b}$. By \eqref{BEQN5} we have $2^{b-1}\equiv0\pmod{2^c}$. So $c\leq b-1$ and $b=d$. Note that the first relation in \eqref{IDEN4} is reduced to $x^{p^d}=z^{2^{b-1}}u^{2^{c-1}}$, so $d=a-1$. Combining these with the inequalities in Lemma~\ref{DISJOINT}(iii)-(iv) we thus obtain $1\leq c\leq b= d=a-1$
\end{subcase}
\end{case}
\end{proof}

\section{Classification}
The following theorem is the main result of the paper.
\begin{theorem}\label{MAIN}
Let $D=(G,x,y)$ be a totall symmetric dessin. If $G$ is a $p$-group of class three, then $G$ is isomorphic to one of the following groups:
\begin{itemize}
\item[\rm(A)] $p$ is odd,
\begin{enumerate}
\item[\rm(i)] $p>3$ and $1\leq c\leq b\leq a,$
\begin{align*}G=\la x,y|&x^{p^a}=y^{p^a}=z^{p^b}=u^{p^c}=v^{p^c}=
[x,u]=[x,v]=\\&[y,u]=[y,v]=1, z:=[x,y], u:=[z,x], v:=[z,y]\ra.
\end{align*}
\item[\rm(ii)] $p=3$ and $1\leq c< b= a$ or $1\leq c\leq b\leq a-1$,
\begin{align*}G=\la x,y|&x^{3^a}=y^{3^a}=z^{3^b}=u^{3^c}=v^{3^c}=
[x,u]=[x,v]=\\&[y,u]=[y,v]=1, z:=[x,y], u:=[z,x], v:=[z,y]\ra.
\end{align*}
\end{enumerate}
\item[\rm(B)]$p=2$,
\begin{enumerate}
\item[\rm(iii)]$1\leq c\leq b\leq a-1,$
\begin{align*}
G=\la x,y|&x^{2^a}=y^{2^a}=z^{2^{b}}=u^{2^c}=v^{2^c}=[x,u]=[x,v]=
[y,u]=[y,v]=1,\\ &z:=[x,y], u:=[z,x], v:=[z,y]\ra.
 \end{align*}
\item[\rm(iv)]$1\leq c\leq b\leq a-2,$
\begin{align*}G=\la x,y|&x^{2^a}=y^{2^a}=z^{2^{b}}=u^{2^c}=v^{2^c}=[x,u]=[x,v]=
[y,u]=[y,v]=1, \\x^{2^{a-1}}&=u^{2^{c-1}},y^{2^{a-1}}=v^{2^{c-1}}, z:=[x,y], u:=[z,x], v:=[z,y]\ra.
 \end{align*}
\item[\rm(v)]$1\leq c\leq  a-2,$
\begin{align*}G=\la x,y|&x^{2^a}=y^{2^a}=z^{2^{a-1}}=u^{2^c}=v^{2^c}=[x,u]=[x,v]=
[y,u]=[y,v]=1,\\ &x^{2^{a-1}}=z^{2^{a-2}}, y^{2^{a-1}}=z^{2^{a-2}}, z:=[x,y], u:=[z,x],v:=[z,y]\ra.
 \end{align*}
\item[\rm(vi)]$1\leq c\leq a-2,$
\begin{align*}G=\la x,y|&x^{2^a}=y^{2^a}=z^{2^{a-1}}=u^{2^c}=v^{2^c}=[x,u]=[x,v]=
[y,u]=[y,v]=1,\\ &x^{2^{a-1}}=z^{2^{a-2}}u^{2^{c-1}}, y^{2^{a-1}}=z^{2^{a-2}}v^{2^{c-1}}, z:=[x,y], u:=[z,x], v:=[z,y]\ra.
 \end{align*}
\end{enumerate}
\end{itemize}
Moreover, each of the above groups underlies a unique regular desssin, and the groups from distinct families, or from the same family but with distinct parameters, are not isomorphic.
\end{theorem}
\begin{proof}
By the assumption we have $G=\la x,y\ra$. Define $z=[x,y]$, $u=[z,x]$ and $v=[z,y]$. Since $G$ is of class three, $G$ is metabelan. So by Lemma~\ref{DISJOINT}(i), $o(x)=o(y)$ and $o(u)=o(v)$. Define $o(x)=p^a$, $o(z)=p^b$ and $o(u)=p^c$. By Lemma~\ref{META3}, $G_3=\la u,v\ra$ and $G_2=\la z,u,v\ra$. Let $H=\la x,z,u,v\ra$. Then we obtain a series of normal subgroups of $G$:
\[
1\leq G_3\leq G_2\leq H\leq G,
\]
where $G_2/G_3=\la zG_3\ra$, $H/G_2=\la xG_2\ra$ and $G/H=\la yH\ra$ are all cyclic groups. Assume that $\la x\ra\cap G_2=\la x^{p^d}\ra$ where $d$ is a nonnegative integer. Then $x^{p^d}=z^lu^mv^n$ where $l,m,n$ are nonnegative integers such that 
\[
0\leq l<p^b \quad\text{and}\quad 0\leq m,n<p^c.
\]
 By Lemma~\ref{INTERSECTION}, $\langle y\rangle\cap H=\langle y\rangle\cap G_2$. Since $\langle y\rangle\cap G_2=\tau(\langle x\rangle\cap G_2)$, we have 
\[
\langle y\rangle\cap H=\tau(\langle x\rangle\cap G_2)=\tau(\langle x^{p^d}\rangle)=\langle y^{p^d}\rangle.
\]
We have $y^{p^d}=\tau(x^{p^d})=\tau(z^lu^mv^n)=z^{-l}v^{-m}u^{-n}.$ By Lemma~\ref{DISJOINT}(iv), $n=0$. Therefore, by Lemma~\ref{DISJOINT}(ii) the group $G$ has a presentation
\begin{equation*}
\begin{aligned}
G=\la x,y\mid x^{p^a}&=y^{p^a}=z^{p^b}=u^{p^c}=v^{p^c}=[x,u]=[x,v]=[y,u]=[y,v]=1, \\ &x^{p^d}=z^lu^m, y^{p^d}=z^{-l}v^{-m},z:=[x,y], u:=[z,x], v:=[z,y]\ra,
\end{aligned}
\end{equation*}
 By filling the numerical conditions proved in Corollary~\ref{ALLCASE} the presentation is simplified to the stated forms.

To show that the group $G$ underlies a unique regular dessin, by Proposition~\ref{FUNDM}(ii) it is sufficient to show that $\Aut(G)$ is transitive on the generating pairs of $G$.  By Burnside's Basis Theorem each generating pair $(x_1,y_1)$ of $G$ can be written as the form
 \[
( x_1,y_1)=(x^{i_1}y^{i_2}z^{i_3}u^{i_4}v^{i_5}, x^{j_1}y^{j_2}z^{j_3}u^{j_4}v^{j_5})
 \]
where
 \begin{align}\label{GEN2}
i_1j_2-i_2j_1\not\equiv0\pmod{p}.
 \end{align}
So it suffices to verify that the assignment $x\mapsto x_1,y\mapsto y_1$ extends to a group automorphism of $G$. Define $z_1=[x_1,y_1]$, $u_1=[z_1,x_1]$ and $v_1=[z_1,y_1]$. Then
\begin{align*}
z_1=&[x_1,y_1]=[x^{i_1}y^{i_2}z^{i_3}u^{i_4}v^{i_5},x^{j_1}y^{j_2}z^{j_3}u^{j_4}v^{j_5}]\\
=&[x^{i_1}y^{i_2}z^{i_3},x^{j_1}y^{j_2}z^{j_3}]\\
=&[x^{i_1}y^{i_2},x^{j_1}y^{j_2}z^{j_3}][z^{i_3},x^{j_1}y^{j_2}z^{j_3}]\\
=&[x^{i_1}y^{i_2},x^{j_1}y^{j_2}][x^{i_1}y^{i_2},z^{j_3}][z^{i_3},x^{j_1}y^{j_2}]\\
=&z^{i_1j_2-i_2j_1}u^{j_2{i_1\choose2}-i_2{j_1\choose2}}v^{i_1i_2j_2-j_1i_2j_2+i_1{j_2\choose2}-j_1{i_2\choose2}}\\
&[x^{i_1},z^{j_3}][y^{i_2},z^{j_3}][z^{i_3},y^{j_2}][z^{i_3},x^{j_1}]\\
=&z^Lu^Mv^N,
\end{align*}
where
\begin{align}
L=&i_1j_2-i_2j_1,\label{FEQN4}\\
M=&j_2{i_1\choose2}-i_2{j_1\choose2}-i_1j_3+i_3j_1,\label{FEQN5}\\ N=&i_1i_2j_2-j_1i_2j_2+i_1{j_2\choose2}-j_1{i_2\choose2}-i_2j_3+i_3j_2.\label{FEQN6}
\end{align}
We also have
\begin{align*}
u_1=&[z_1,x_1]=[z^Lu^Mv^N,x^{i_1}y^{i_2}z^{i_3}u^{i_4}v^{i_5}]=[z^L,x^{i_1}y^{i_2}z^{i_3}]\\
=&[z^L,x^{i_1}y^{i_2}]=[z^L,y^{i_2}][z^L,x^{i_1}]=u^{Li_1}v^{Li_2},
\end{align*}
and
\begin{align*}
v_1:=&[z_1,y_1]=[z^Lu^Mv^N,x^{j_1}y^{j_2}z^{j_3}u^{j_4}v^{j_5}]=[z^L,x^{j_1}y^{j_2}z^{j_3}]\\
=&[z^L,x^{j_1}y^{j_2}]=[z^L,y^{j_2}][z^L,x^{j_1}]=u^{Lj_1}v^{Lj_2}.
\end{align*}

 Clearly, in each case \[z_1^{p^b}=u_1^{p^c}=v_1^{p^c}=[x_1,u_1]=[x_1,v_1]=[y_1,u_1]=[y_1,v_1]=1.\] 
In what follows we verify that $x_1$ and $y_1$ fulfil the remaining defining relations case by case.

\begin{case}[i] $p>3$, $1\leq c\leq b\leq d=a$.\par
 In this case we need to show that $x_1^{p^a}=1$ and $y_1^{p^a}=1$.  By Corollary~\ref{META6} and Corollary~\ref{META7} we have
\begin{align*}
x_1^{p^a}=&(x^{i_1}y^{i_2}z^{i_3}u^{i_4}v^{i_5})^{p^a}=\big((x^{i_1}y^{i_2})z^{i_3}\big)^{p^a}u^{i_4p^a}v^{i_5p^a}\nonumber\\
=&(x^{i_1}y^{i_2})^{p^a}z^{i_3p^a}[x^{i_1}y^{i_2},z]^{-i_3{p^a\choose2}}u^{i_4p^a}v^{i_5p^a}\\
=&x^{i_1p^a}y^{i_2p^a}z^{R}u^Sv^T=z^{R}u^Sv^T,
\end{align*}
where
\begin{align}
R=&i_3p^a-i_1i_2{p^a\choose2},\label{FEQN1}\\
S=&-i_2{i_1\choose2}{p^a\choose2}-i_1^2i_2{p^a\choose3}+i_1i_3{p^a\choose2}+i_4p^a,\label{FEQN2}\\
T=&-i_1{i_2\choose2}{p^a\choose2}+i_1i_2^2\big({p^a\choose3}+(1-p^a){p^a\choose2}\big)+i_2i_3{p^a\choose2}+i_5p^a.\label{FEQN3}
\end{align}
Recall that $p>3$ is odd, so $p^a|{p^a\choose2}$ and $p^a|{p^a\choose3}$. So $p^a\mid R$, $p^a\mid S$ and $p^a\mid T$, and hence $x_1^{p^a}=1$. Similarly, $y_1^{p^a}=1$.
\end{case}

\begin{case}[ii] $p=3$ and $1\leq c< b=d=a$ or $1\leq c\leq b<e=d=a$.\par
Note that $3^a|{3^a\choose2}$ and $3^{a-1}|{3^a\choose3}$. Using similar arguements as in (i) it is straightforward to verify that $x_1^{3^a}=1$ and $y_1^{3^a}=1$.
\end{case}

\begin{case}[iii]$p=2$ and $1\leq c\leq b\leq d-1=a-1$.\par
Note that $2^{a-1}|{2^a\choose 2}$ and $2^{a}|{2^a\choose 3}$. Using similar arguements as in (i) it is straightforward to verify that $x_1^{2^a}=1$ and $y_1^{2^a}=1$.
\end{case}

\begin{case}[iv] $p=2$ and $1\leq c\leq b\leq e-1=a-2$.\par
 Using similar arguements as in (i) it is easy to prove that $x_1^{2^a}=y_1^{2^a}=1.$ Since $2^{a-2}|{2^{a-1}\choose2}$ and $2^{a-1}|{2^{a-1}\choose3}$, we have \[x_1^{2^{a-1}}=x^{i_12^{a-1}}y^{i_22^{a-1}}=u^{i_12^{c-1}}v^{i_22^{c-1}}\] and \[u_1^{2^{c-1}}=u^{Li_12^{c-1}}v^{Li_22^{c-1}}=u^{i_12^{c-1}}v^{i_22^{c-1}}.\] So $x_1^{2^{a-1}}=u_1^{2^{c-1}}$. Similarly, $y_1^{2^{a-1}}=v_1^{2^{c-1}}$.
\end{case}

\begin{case}[v]$p=2$ and $1\leq c\leq b-1=d-1=a-2$.\par
  Using similar arguements as in (i) it is easy to prove that $x_1^{2^a}=y_1^{2^a}=1.$ Since $2^{a-2}|{2^{a-1}\choose2}$ and $2^{a-1}|{2^{a-1}\choose3}$, we have
\begin{align*}
x_1^{2^{a-1}}&=x^{i_12^{a-1}}y^{i_22^{a-1}}z^{-i_1i_22^{a-2}(2^{a-1}-1)}\\
&=z^{(i_1+i_2-i_1i_2(2^{a-1}-1))2^{a-2}}=z^{2^{a-2}}
\end{align*}
and
 \[z_1^{2^{a-2}}=z^{(i_1j_2-i_2j_1)2^{a-2}}=z^{2^{a-2}}.\]
Hence $x_1^{2^{a-1}}=z_1^{2^{a-2}}$. Similarly, $y_1^{2^{a-1}}=z_1^{2^{a-2}}$.
\end{case}

\begin{case}[vi]$p=2$ and $1\leq c\leq b-1=d-1=a-2$.\par
  Using similar arguements as in (i) it is easy to prove that $x_1^{2^a}=y_1^{2^a}=1.$ Since $2^{a-2}|{2^{a-1}\choose2}$ and $2^{a-1}|{2^{a-1}\choose3}$, we have
\begin{align*}
x_1^{2^{a-1}}&=x^{i_12^{a-1}}y^{i_22^{a-1}}z^{-i_1i_22^{a-2}(2^{a-1}-1)}=z^{(i_1+i_2-i_1i_2(2^{a-1}-1))2^{a-2}}u^{i_12^{c-1}}v^{i_22^{c-1}}\\
&=z^{2^{a-2}}u^{i_12^{c-1}}v^{i_22^{c-1}}
\end{align*}
and \[z_1^{2^{a-2}}u_1^{2^{c-1}}=z^{(i_1j_2-i_2j_1)2^{a-2}}u^{Li_12^{c-1}}v^{Li_22^{c-1}}=z^{2^{a-2}}u^{i_12^{c-1}}v^{i_22^{c-1}}.\]
Hence $x_1^{2^{a-1}}=z_1^{2^{a-2}}u_1^{2^{c-1}}$. Similarly, $y_1^{2^{a-1}}=z_1^{2^{a-2}}v_1^{2^{c-1}}$.
\end{case}

Finally, to prove the isomorphism relations between the groups we summarize the invariant types of $G'$ and $G^{\mathrm{ab}}=G/G'$ as follows.
\[
\begin{array}{llll}
\text{Case} & G' &  G^{\mathrm{ab}} & \text{Condition}\\
\hline
{\rm(i)} &C_{p^c}\times C_{p^c}\times C_{p^b} & C_{p^a}\times C_{p^a} & 1\leq c\leq b\leq a\\
{\rm(ii)} &C_{3^c}\times C_{3^c}\times C_{3^b} & C_{3^a}\times C_{3^a} & \text{$1\leq c< b\leq a$ or $1\leq c\leq b\leq a-1$}\\
{\rm(iii)} &C_{2^c}\times C_{2^c}\times C_{2^b} & C_{2^a}\times C_{2^a} & 1\leq c\leq b\leq a-1\\
{\rm(iv)} &C_{2^c}\times C_{2^c}\times C_{2^b} & C_{2^{a-1}}\times C_{2^{a-1}} & 1\leq c\leq b\leq a-2\\
{\rm(v)} &C_{2^c}\times C_{2^c}\times C_{2^{a-1}} & C_{2^{a-1}}\times C_{2^{a-1}} & 1\leq c\leq a-2\\
{\rm(vi)}&C_{2^c}\times C_{2^c}\times C_{2^{a-1}} & C_{2^{a-1}}\times C_{2^{a-1}} &1\leq c\leq a-2
\end{array}
\]
It is clear that the groups are pairwise non-isomorphic, except possibly the groups from (v) and (vi). But each of the groups has a unique presentation, so by comparing the presentations we find that every group from (v) is not isomorphic to any group from (vi), as claimed.
\end{proof}

\begin{corollary}\label{Extra}
Let $G$ be a group from Theorem~\ref{MAIN}. Then the size of the group $G$ and its automorphism group $\Aut(G)$, and the type and genus of the associated totally symmetric dessin $D$ with $\Aut(D)\cong G$ are summarized as follows.
\[
\begin{array}{l|llll}\label{TT}
$G$ & |G|              &  |\Aut(G)|             & \text{Type of $D$}&\text{Genus of $D$} \\
\hline
{\rm (i)} &p^{2(a+c)+b}& p^{4a+2b+4c-3}(p+1)(p-1)^2 & (p^a,p^a,p^a)      &p^{a+b+2c}(p^a-3)/2+1\\
{\rm (ii)} &3^{2(a+c)+b}& 2^4\cdot3^{4a+2b+4c-3} & (3^a,3^a,3^a)      &3^{a+b+2c+1}(3^{a-1}-1)/2+1\\
{\rm (iii)} &2^{2(a+c)+b}& 3\cdot2^{4a+2b+4c-3} & (2^a,2^a,2^a)      &2^{a+b+2c-1}(2^a-3)+1\\
{\rm (iv)} &2^{2(a+c)+b-2}& 3\cdot2^{4a+2b+4c-7} & (2^a,2^a,2^a)      &2^{a+b+2c-3}(2^a-3)+1\\
{\rm (v)} &2^{3a+2c-3}&  3\cdot2^{6a+4c-9} & (2^a,2^a,2^a)      &2^{2a+2c-4}(2^a-3)+1\\
{\rm (vi)} &2^{3a+2c-3}& 3\cdot2^{6a+4c-9} & (2^a,2^a,2^a)      &2^{2a+2c-4}(2^a-3)+1\\
\end{array}
\]
\end{corollary}

\begin{remark}
In \cite[Research problems and themes I 35(a)]{BJ2008}, Berkovich and Janko posed a problem of studying $p$-group $G$ with $|G:\Phi(G)|=p^d$ satisfying 
\[
|\Aut(G)|=(p^d-1)(p^d-p)\ldots(p^d-p^{d-1})|\Phi(G)|^d,
\]
that is, $G$ is a $d$-generator $p$-group whose automorphism group $\Aut(G)$ acts transitively on its generating $d$-tuples. Theorem~\ref{MAIN} and Corollary~\ref{Extra} is an answer to this problem for $d=2$ and $c(G)=3$.
\end{remark}

\section*{Acknowledgement}
The first and third author were supported by Zhejiang Provincial Natural Science Foundation of China under Grant No. LY16A010010,  and by Scientific Research Foundation(SRF) of Zhejiang Ocean University (21065014015, 21065014115). The second author was supported by the following grants: VEGA 1/0150/14, APVV-0223-10, and the grant APVV-ESF-EC-0009-10 within the EUROCORES Programme EUROGIGA (Project GReGAS) of the European Science Foundation and the Slovak-Chinese bilateral grant APVV-SK-CN-0009-12.


\end{document}